\documentclass{article}
\usepackage[utf8]{inputenc}
\usepackage[T2A]{fontenc}
\usepackage[english]{babel}
\inputencoding{utf8}

\usepackage{amsmath, amssymb, amsthm, nicefrac, mathtools, url}

\usepackage[
    pdfpagemode=UseNone,
    bookmarks=true,
    bookmarksopen=true,
    pdfpagelayout=SinglePage,
    pdfstartview={XYZ null null 1},
    pdfhighlight=/P,
    colorlinks=true,
    linkcolor=blue,
    citecolor=blue,
    urlcolor=blue
    ]{hyperref}
\usepackage{blkarray}
\usepackage{float}

\advance\textwidth 20mm  \advance\oddsidemargin -10mm
\advance\textheight 20mm \advance\topmargin -15mm

\theoremstyle{plain}
\newtheorem{theorem}{Theorem}
\newtheorem{lemma}[theorem]{Lemma}
\newtheorem{proposition}{Proposition}

\newtheorem*{definition}{Definition}
\newtheorem{remark}{Remark}

\DeclareMathOperator{\diag}{diag}
\DeclareMathOperator{\Mat}{Mat}

\newcommand{\brackets}[1]{\left( #1 \right)}

\newcommand{\LieBrackets}[1]{\left[ #1 \right]}

\DeclareMathOperator{\PII}{P_{2}}
\DeclareMathOperator{\PIV}{P_{4}}

\newcommand{\PPainleve}{Painlev{\'e}}
\newcommand{\PIVn}[1]{\text{P}_4^{ #1 }}

\mathtoolsset{showonlyrefs,showmanualtags}

\title{On matrix Painlev\'e-4 equations. \\ Part 1: Painlev\'e--Kovalevskaya test}

\date{13 June 2021}

\author{I.A. Bobrova\thanks{National Research Univerisity ``Higher School of Economics'', Moscow, Russian Federation.},~ V.V. Sokolov\thanks{L.D.~Landau Institute for Theoretical Physics, Chernogolovka, Russian Federation.} \thanks{Federal University of ABC, Santo Andr\'e, Sao Paulo, Brazil. E-mail: vsokolov@landau.ac.ru}}

\usepackage{tikz}

\begin{document}
\maketitle

\begin{abstract}
Using the Painlev\'e--Kovalevskaya test, we find several new matrix generalizations of the Painlev\'e-4 equation. Some limiting transitions reduce them to known matrix \PPainleve-2 equations.
\medskip

\end{abstract}

%-------------------------------------------------------------------------------

\section{Introduction}

\qquad The theory of the Painlev\'e equations is a wonderful part of the analytic theory of ODEs. In addition to their purely mathematical value, the Painlev\'e equations have arisen in a variety of applications including
random matrix theory, topological field theory, plasma physics, quantum gravity, and general relativity.  They also describe self-similar modes
of nonlinear PDEs.

Some models of mathematical physics are described by differential equations with matrix variables. In this case, it is natural to expect the appearance of matrix generalizations of the Painlev\'e equations similar to the scalar ones.

In the paper \cite{Adler_Sokolov_2020_1}, it was shown that the second Painlev\'e equation admits at least three different matrix intgerable generalizations. All of them satisfy the matrix \PPainleve--Kovalevskaya test and possess isomonodromic Lax pairs.  Сertainly, a similar variety should be expected for other Painlev\'e equations. Although the literature on such generalizations is quite rich, the question remains as to how many kinds of them there are and, in particular, how matrix constant coefficients can be included into the equations.

In this paper, we consider matrix generalizations of the Painlev\'e equation $\PIV$ 
\begin{equation}\label{P4scal}
 y''= \frac{y'^2}{2\, y}  +\frac{3}{2} y^3+4z y^2 +2(z^2-\gamma) y+\frac{\delta}{y}.
\end{equation}
Here $'$ means the derivative with respect to  the variable $z$. 

It is known that equation \eqref{P4scal} is equivalent to the system
\begin{equation} \label{syscal}
\left\{
\begin{array}{lcl}
u' &=& -u^2 + 2 u v - 2 z u + c_1 , 
\\[2mm]
v' &=& -v^2 + 2 u v + 2 z v + c_2,
\end{array}
\right.
\end{equation}
where \,\, $y=u$, \,\, $\gamma=1 + \frac{1}{2} c_1 - c_2$, \,\, $\delta = - \frac{1}{2} c_1^2$. 

Using the matrix Painlev\'e test, we look for integrable matrix generalizations of system \eqref{syscal} of the form 
\begin{equation}\label{3rootTail}
\left\{
\begin{array}{lcl}
u' &=&  - u^2 + 2 \,u v  + \alpha (u v - v u) - 2 z u + b_1 u+u b_2+b_3 v+v b_4+b_5
\\ [2mm]
v' &=& -  v^2 + 2 \, v u + \beta (v u - u v) + 2 z v+c_1 v+v c_2 +c_3 u +u c_4+c_5,
\end{array}   
\right.
\end{equation}
where the coefficients $\alpha$, $\beta$ are scalar and others are constant matrices. Here and below, we consider matrices over the field $\mathbb{C}$.

We divide our research into two steps. We first investigate the principle homogeneous part of system \eqref{3rootTail} (see Section \ref{sec:Painlevetest}). It defines basic ingredients for the \PPainleve--Kovalevskaya analysis of system \eqref{3rootTail} such as the structure of leading coefficients of the formal Laurent solutions, resonances, dimensions of the corresponding eigenspaces and so on. At the second stage, admissible matrix coefficients $b_i$ and $c_i$ are found.

The principal homogeneous part of system \eqref{syscal} has the form 
\begin{equation}\label{Ptype}
\left\{
\begin{array}{lcl}
u' &=& -u^2 + 2 u v, \\[2mm]
v' &=& -v^2 + 2 u v.
\end{array}
\right.
\end{equation}
Non-abelian generalizations of system \eqref{Ptype} of the form
\begin{equation}\label{3rootN}
\left\{
\begin{array}{lcl}
u' &=&  - u^2 + 2 \,u v  + \alpha (u v - v u),
\\ [2mm]
v' &=& -  v^2 + 2 \, v u + \beta (v u - u v),
\end{array}   
\right.
\end{equation}
where $\alpha$, $\beta \in {\mathbb C}$, were studied in \cite{SW2}. It turns out that such a system possesses a hierarchy of polynomial infinitesimal symmetries if $(\alpha, \beta)$ is one of the dots 
 in the following figure:

{\color{black}
\begin{figure}[H]
    \centering
    \scalebox{1.1}{\input{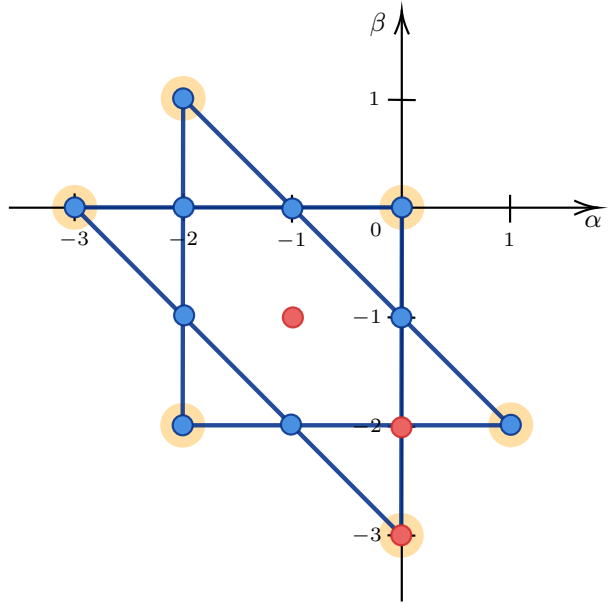}}
    \caption{We choose the points highlighted in red as representatives of  three orbits of the symmetry group action.}
    \label{pic:albetplane}
\end{figure}
}
Let us denote by $\Sigma$ the set of thirteen integer points shown in Figure \ref{pic:albetplane}.

\begin{remark} Another gauge equivalent version of systems \eqref{3rootN} is:
\begin{equation}\label{3rootN1}
    \left\{
    \begin{array}{lcl}
    u' &=&  v^2 + \alpha (u v - v u),
    \\ [2mm]
    v' &=& u^2  + \beta (v u - u v),
    \end{array}   
    \right.
\end{equation}
where $\alpha$, $\beta \in {\mathbb C}$. The existence of infinitesimal symmetries leads to $\beta=\bar \alpha$, where $\alpha$'s are defined by the root system of $G_2$-type:
\begin{figure}[H]
    \centering
    \scalebox{1.1}{\input{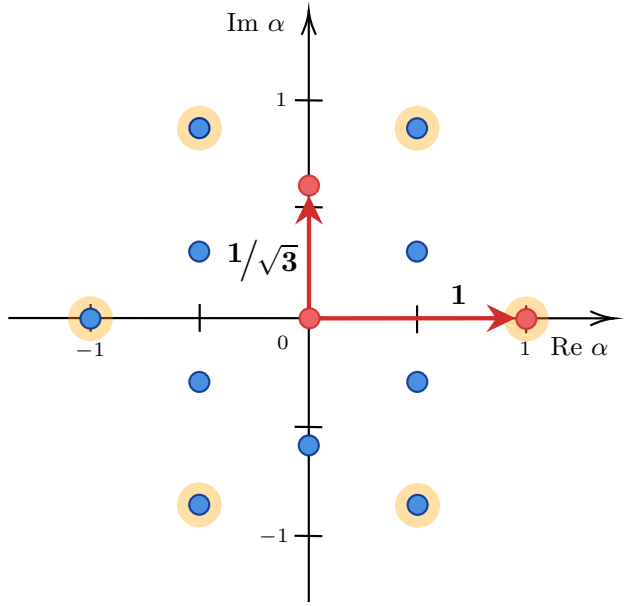}}
    \caption{$\alpha$-plane}
    \label{pic:G2}
\end{figure}
\end{remark}

The involutions
\begin{gather}
    \label{trr1}
    (u,\,v) \mapsto (v,\,u),
    \\[2mm]
    \label{trr2}
    (u,\,v) \mapsto (u^T,\,v^T),  
\end{gather}
and
\begin{equation}
    (u,v) \mapsto (-u,v-u),
\end{equation}
where $T$ means the matrix transposition, preserve the class of such systems \eqref{3rootN} changing the parameters in the following way: 
\begin{align} 
    (\alpha,\beta) 
    &\mapsto (\beta, \alpha), 
    &
    (\alpha,\beta) 
    &\mapsto (- \alpha - 2,   - \beta - 2), 
    &
    (\alpha,\beta) 
    &\mapsto (\alpha,   - \alpha - \beta - 3),
\end{align}
respectively.
The group generated by these three involutions of the $(\alpha, \beta)$-plane is isomorphic to the dihedral group $D_{12}$ of symmetries of a regular 6-gon. The point $(-1,-1)$, the six point marked in Figure \ref{pic:albetplane} with the dots surrounded by an orange rim, and the remaining six points form three orbits of the group action.

In Section \ref{sec:Painlevetest}, we study Painlev\'e properties of matrix systems \eqref{3rootN}. It turns out that under some assumptions (see Theorem \ref{theo4}), they satisfy the matrix \PPainleve--Kovalevskaya test \cite{Balandin_Sokolov_1998} only for points $(\alpha, \beta)$ shown in Figure \ref{pic:albetplane}. 

For each homogeneous system \eqref{3rootN}, where  $(\alpha, \beta)\in \Sigma$   we find in Section \ref{sec:P4tailedsystems} linear terms that can be added to the right hand side so that the resulting system \eqref{3rootTail} still satisfies the matrix \PPainleve--Kovalevskaya test. 

Transformations \eqref{trr1}, \eqref{trr2}
and
\begin{equation}\label{trr3}
    (u,\,v) \mapsto (-u,\,v-u - 2 z) 
\end{equation}
map\footnote{For transformations \eqref{trr2} and \eqref{trr3} one should use the scaling $z\mapsto-i z,\, u\mapsto i u, \, v \mapsto i v$ to correct the signs before $2 z$ in the resulting system \eqref{3rootTail} and after transformation \eqref{trr3} also a shift of $u$ and $v$ is needed to bring the result to the form \eqref{3rootTail}. \label{footnote:scal}} a system \eqref{3rootTail} to another system of the same form but change the coefficients.

One more class of transformations we use in this paper is given by the formula
\begin{equation} \label{tranconj}
    u 
     \mapsto e^{z K} 
    \brackets{u + Q_1} 
    e^{- z K},
    \qquad 
    v \mapsto e^{z K} 
    \brackets{v + Q_2} 
    e^{- z K},
\end{equation}
where $K$ and $Q_i$ are constant matrices.
In general, a transformation \eqref{tranconj} takes a system of the from \eqref{3rootTail} outside of the class of such systems. But in very particular cases transformations \eqref{tranconj} can be applied. 

Let us formulate the main result of Section \ref{sec:P4tailedsystems} 
\begin{theorem} Any system \eqref{3rootTail}  that satisfies the matrix \PPainleve--Kovalevskaya test can be reduced by transformations \eqref{trr1} -- \eqref{tranconj} to one of the following:
\begin{align} \label{eq:case1tail_can}
    \tag*{$\PIVn{0}$}
    &\left\{
    \begin{array}{lcl}
         u' 
         &=& - u^2 + u v + v u - 2 z u 
         + h u 
         + \gamma_1 \, \mathbb{I},
         \\[2mm]
         v' 
         &=& - v^2 + v u + u v + 2 z v
         - v h
         + \gamma_2 \, \mathbb{I},
    \end{array}
    \right.
    \hspace{1.15cm}
\end{align}
\begin{align} 
    \label{eq:case3tail_can}
    \tag*{$\PIVn{1}$}
    &\left\{
    \begin{array}{lcl}
         u' 
         &=& - u^2 + 2 u v - 2 z u 
         + h,
         \\[2mm]
         v' 
         &=& - v^2 + 2 u v + 2 z v
         + h + \gamma \, \mathbb{I},
    \end{array}
    \right.
    \hspace{2.15cm}
\end{align}
\begin{align} 
    \label{eq:case5tail_can}
    \tag*{$\PIVn{2}$}
    &\left\{
    \begin{array}{lcl}
         u' 
         &=& - u^2 + 2 u v - 2 z u 
         + h_2,
         \\[2mm]
         v' 
         &=& - v^2 + 3 u v - v u + 2 z v
         + h_1 u 
         + 2 h_2 + \gamma \, \mathbb{I}.
    \end{array}
    \right.
\end{align}
Here $\gamma$, $\gamma_i \in \mathbb{C}$, $h$ is an arbitrary matrix and two constant matrices $h_1$, $h_2$ are connected by the commutation relation $[h_2, h_1] = - 2 h_1$. 
\end{theorem}

The parameter $\alpha$ and the coefficients $b_3$, $b_4$ are equal to zero in systems \ref{eq:case3tail_can} and  \ref{eq:case5tail_can}. Therefore, the unknown variable $v$ can be eliminated and as a result we arrive at two matrix $\text{P}_4$ equations 
of the form 
\begin{equation}\label{P4mat}
\begin{array}{c}
 \displaystyle 
 y''
 = \frac{1}{2}(y' + k_1)
 \, y^{-1}
 \, (y' + k_2)  
 + \frac{3}{2} y^3
 + \kappa [y',y]
 + 4 z y^2
 + y k_{3} y
 + k_{4} y 
 + y k_{5} 
 + 2 z^2 y
 \end{array}
\end{equation}
for $y(z)=u(z)$.
The coefficients $k_i$ are expressed in terms of the corresponding system \eqref{3rootTail}  by the formulas 
\begin{align*}
    &\begin{aligned}
    \kappa
    \,\,
    &= \beta + \frac32,
    &&&
    k_1 
    &= - k_2
    = b_5,
    &&&
    k_3
    &= 2 c_3,
    \end{aligned}
    \\
    &\begin{aligned}
    k_4
    &= - 2 - \brackets{\beta + \frac32} b_5,
    &&&
    k_5
    &= 2 c_5 + \brackets{\beta + \frac12} b_5.
    \end{aligned}
\end{align*}
One of these equations is equivalent to the matrix $\text{P}_4$ equation presented in \cite{Adler_Sokolov_2020_1}.
In the case of scalar coefficients $k_i$ equations are equivalent to those found in  \cite{adler2020}.

System \ref{eq:case1tail_can} can be derived from an obvious matrix generalization of the dressing chain  with $N=3$ investigated in  \cite{veselov1993}.

In order to express the variable $v$ from the first equation of this system, one needs to invert the operator $P_u:\, P_u (v) = u v + v u$ and, therefore, $u''$ is not a rational function in the variables $u$ and $u'$.

\begin{remark} 
The principal homogeneous part of the \text{\rm\ref{eq:case1tail_can}} system can be reduced to the form~\eqref{3rootN1}. After renaming the coefficients, the resulting system is given by
\begin{align} \label{eq:case1tail_can_2}
    &\left\{
    \begin{array}{lcl}
         u' 
         &=& v^2 - 2 z u 
         + h u 
         + \gamma_1 \, \mathbb{I},
         \\[2mm]
         v' 
         &=& u^2 + 2 z v
         - v h
         + \gamma_2 \, \mathbb{I},
    \end{array}
    \right.
\end{align}
where $h$ is an arbitrary matrix and $\gamma_i \in \mathbb{C}$.
\end{remark}

In Section \ref{degeneracies} we find limiting transitions from the matrix \PPainleve-4 systems  to the matrix \PPainleve-2 equations found in \cite{Adler_Sokolov_2020_1}.

\section{Matrix Painlev\'e--Kovalevskaya test for homogeneous \texorpdfstring{\\}{} systems \eqref{3rootN}}
\label{sec:Painlevetest}

The scalar  system \eqref{Ptype} satisfies the \PPainleve--Kovalevskaya test. Namely, it has formal Laurent solutions of the following three types:   
$$
{\bf 1}: \quad  u = -\frac{1}{z-z_0}+O(1), \qquad  v = -\frac{1}{z-z_0}+O(1);
$$
$$
{\bf 2}: \quad  u = \frac{1}{z-z_0}+O(1);  \qquad  v = O(1); \qquad \quad 
{\bf 3}: \quad  u =O(1), \qquad  v =  \frac{1}{z-z_0}+O(1)
$$
that contains, apart from $z_0$, another arbitrary constant. For example, a solution of type {\bf 1} has the form 
$$
u = -\frac{1}{z-z_0}+ \sigma (z-z_0)^2 - \frac{3}{7} \sigma^2  (z-z_0)^5 + \cdots,  
\qquad 
v = -\frac{1}{z-z_0} - \sigma (z-z_0)^2 - \frac{3}{7} \sigma^2  (z-z_0)^5 + \cdots,
$$
where $\sigma$ is arbitrary.

Let us find out when system \eqref{3rootN} possesses a formal solution of the form
\begin{equation}\label{y}
    u
    = \frac{p}{z-z_0}+x_0+x_1(z-z_0)+\cdots\,,
    \qquad 
    v
    =\frac{q}{z-z_0}+y_0+y_1(z-z_0)+\cdots\,,
\end{equation}
where $p$, $q$, $x_j$, $y_j \in \Mat_n (\mathbb{C})$, \,\, $z_0 \in \mathbb C$, containing the maximum possible number $2 n^2$ of arbitrary constants. We will call such a solution {\it maximal}. 

Substituting the series \eqref{y} into  \eqref{3rootN}, we obtain the following recurrence relations for their coefficients:
\begin{gather}
    \label{eq:resconditions}
    - p^2 
    + 2 p q 
    + \alpha \LieBrackets{p, q}
    + p
    = 0,
    \qquad
    - q^2 
    + 2 q p 
    + \beta \LieBrackets{q, p}
    + q
    = 0,
    \\[3mm]
    \begin{aligned}
    \label{eq:coeffconditions}
    -  p x_k - x_k p
    + 2 \brackets{p y_k + x_k q}
    + \alpha \brackets{
    \LieBrackets{p, y_k}
    + \LieBrackets{x_k, q}
    }
    - k \, x_k
    &= f_{\alpha} \brackets{x_{k-1}, y_{k-1}},
    \\[2mm]
    - q y_k - y_k q
    + 2 \brackets{q x_k + y_k p}
    + \beta \brackets{
    \LieBrackets{q, x_k}
    + \LieBrackets{y_k, p}
    }
    - k \, y_k 
    &= f_{\beta} \brackets{y_{k-1}, x_{k-1}},
    \end{aligned}
\end{gather}
where  $f_{\gamma} \brackets{X_{-1}, Y_{-1}}\stackrel{def}{=}0$ and for each  $ k \in \mathbb{N}$ the function $f_{\gamma} \brackets{X_k, Y_k}$ is defined by the formula  
\begin{align}
    \label{eq:rhsFgamma}
    f_{\gamma} \brackets{X_k, Y_k}
    &\stackrel{def}{=} \sum_{l = 0}^{k}
    \brackets{
    \dfrac{1}{2} (X_l X_{k - l}+X_{k - l} X_l)
    - 2 X_{l} Y_{k - l}
    - \gamma \LieBrackets{X_{l}, Y_{k - l}}
    }.
\end{align}

The relations \eqref{eq:coeffconditions} can be written as 
\begin{gather}
    \label{eq:matrixformofconditions}
    \brackets{
    \mathcal{L}
    - k \, \mathbb{I}
    }
    \begin{pmatrix}
    x_k \\[1mm]
    y_k
    \end{pmatrix}
    = 
    \begin{pmatrix}
    f_{\alpha} \brackets{x_{k-1}, y_{k-1}} \\[1mm]
    f_{\beta} \brackets{y_{k-1}, x_{k-1}}
    \end{pmatrix}
    ,
    \quad 
    k \in \mathbb{Z}_{\geq 0}
    ,
\end{gather}
where the operator $\mathcal{L} : \Mat_n \brackets{\mathbb{C}} \oplus \Mat_n \brackets{\mathbb{C}} \to \Mat_n \brackets{\mathbb{C}} \oplus \Mat_n \brackets{\mathbb{C}}$ acts according to the rule
\begin{gather} \label{eq:Loperator}
    \mathcal{L}
    \begin{pmatrix}
    X \\[1mm]
    Y
    \end{pmatrix}
    =
    \begin{pmatrix}
    - p X - X p
    + 2 \brackets{p Y + X q}
    + \alpha \brackets{
    \LieBrackets{p, Y}
    + \LieBrackets{X, q}
    }
    \\[1mm]
    -  q Y - Y q
    + 2 \brackets{q X + Y p}
    + \beta \brackets{
    \LieBrackets{q, X}
    + \LieBrackets{Y, p}
    }
    \end{pmatrix}
    .
\end{gather}

\subsection{System of matrix quadratic equations \eqref{eq:resconditions}}

Let us consider the system \eqref{eq:resconditions} for the residues $p$ and $q$.
Commuting each of the equations of system \eqref{eq:resconditions} with $p$ and $q$, we get 4 linear equations with constant coefficients with respect to the unknowns $p [p,q]$, $[p,q] p$, $q [p,q]$, $[p,q] q$, $[p,q]$. The determinant of the matrix consisting of the coefficients at the first four unknowns is equal to $4 \Delta$, where 
\begin{equation*}
    \Delta 
    = \alpha^2+\beta^2+\alpha \beta + 3 (\alpha+\beta+1).
\end{equation*}

If $ \Delta=0$, then it follows from the system that $[p, q] = 0$.

In the case $\Delta\ne 0$, solving the system, we get
\begin{align}\label{mainR}
    p [p,q] 
    &= \mu_1 [p,q], 
    &
    [p,q] p 
    &= \mu_2 [p,q], 
    &
    q [p,q] 
    &= \mu_3 [p,q],
    &
    [p,q] q 
    &= \mu_4 [p,q],
\end{align}
where 
\begin{equation}\label{muu}
\begin{aligned}
    \mu_1
    &=-\dfrac{\alpha (3+\alpha+2 \beta)}{2 \Delta}, 
    &
    \mu_2
    &=-\dfrac{(2+\alpha) (3+\alpha+2 \beta)}{2 \Delta},
    \\[3mm]
    \mu_3
    &=-\dfrac{\beta (3+\beta+2 \alpha)}{2 \Delta}, 
    & 
    \mu_4
    &=-\dfrac{(2+\beta) (3+\beta+2 \alpha)}{2 \Delta}.
\end{aligned}
\end{equation}
Relations \eqref{mainR} imply $p q [p,q]=q p [p,q]=\mu_1 \mu_3 [p,q]$ and hence  $[p,q]^2=0.$  
\begin{remark} From \eqref{mainR} it follows that the vector space spanned by $p, q, [p,q]$ is a Lie algebra whose square is one-dimensional.
\end{remark}

\begin{proposition}\label{prop2}  Let $(p,q)$ be a solution of system \eqref{eq:resconditions} such that  $[p, q]=0.$ Then the matrices $p$ and $q$ are simultaneously diagonalizable.
\end{proposition}
\begin{proof} Denote $P_1\stackrel{def}{=}-p^2+ 2 p q+p=0$, $P_2\stackrel{def}{=}-q^2+ 2 p q+q=0$. It is easy to check that 
\begin{equation*}
0 = P_1 \left(q - \frac{3}{2} p- \frac{3}{2}\right) + 2 p P_2 = \frac{3}{2} \left(p^3-p\right).
\end{equation*}
From this equation for $p$ it follows that the Jordan form of $p$ is diagonal. Similarly, the matrix $q$ is diagonalizable. If two matrices are diagonalizable and commute, then they are diagonalizable simultaneously. Indeed, $q$ acts in a diagonalizable way on the eigenspaces of the matrix $p$. Choosing a bases in these eigenspaces consisting of the eigenvectors of the matrix $q$, we give rise to a basis in which both matrices are diagonal.
\end{proof}

Suppose that the matrices $p$ and $q$ are diagonal  (see Proposition \ref{prop2}). Then from \eqref{eq:resconditions} it follows that one may set 
\begin{align}
    \label{eq:resJNF}
    p
    &= 
    \diag \left(
    - \mathbb{I}_{k_1},
    \mathbb{I}_{k_2},
    0_{k_3},
    0_{k_4}
    \right),
    &
    q
    &= 
    \diag
    \left(
    - \mathbb{I}_{k_1},
    0_{k_2},
    \mathbb{I}_{k_3},
    0_{k_4}
    \right),
\end{align}
where $k_1 + k_2 + k_3 + k_4 = n$. 

Consider now the case $[p,q]\ne 0.$ 
Denote by $\Sigma_0$ the collection of all integer points $(\alpha, \beta)$, shown in Figure \ref{pic:albetplane}, with the exception of the point $(-1,-1)$. 

\begin{proposition}\label{prop1} Suppose that there exists a solution $p,q$ of system \eqref{eq:resconditions} such that $[p,q]\ne 0$, then  $(\alpha, \beta)\in \Sigma_0$. 
\end{proposition}
\begin{proof} Multiplying each of equations \eqref{eq:resconditions} by the commutator $[p, q]$ from the left  and from the right and using relations \eqref{mainR}, we derive that the equations 
\begin{align*}
    -\mu_1^2+2 \mu_1 \mu_3+\mu_1
    &=0, 
    &
    -\mu_3^2+2 \mu_1 \mu_3+\mu_3
    &=0,
\end{align*} 
\begin{align*}
    -\mu_2^2+2 \mu_2 \mu_4+\mu_2
    &=0, 
    &
    -\mu_4^2+2 \mu_2 \mu_4+\mu_4
    &=0
\end{align*}
must be satisfied. Substituting the values \eqref{muu} for $\mu_i$ and solving the resulting system with respect to $\alpha, \beta$, we arrive at the condition $(\alpha, \beta)\in \Sigma_0$.
\end{proof}

\begin{proposition}\label{prop22} Let $n=2$. Then for each point $(\alpha,\beta)\in \Sigma_0$ there exists a unique up to a conjugation solution of system \eqref{eq:resconditions} such that $[p,q]\ne 0.$ The solution can be normalized by the conditions
\footnote{Such a solution is unique up to a conjugation by means of a upper-triangular matrix with units on the diagonal.} 
$$
p= \begin{pmatrix}
    \mu_1 & X \\
    0 & \mu_2
    \end{pmatrix}, \qquad 
q= \begin{pmatrix}
    \mu_3 & Y \\
    0 & \mu_4
    \end{pmatrix}, \qquad [p,\,q] =  \begin{pmatrix}
    0 & 1 \\
    0 & 0
    \end{pmatrix},
$$
where $\mu_i$ are defined by formulas \eqref{muu}. The pairs $(X, Y)$ can be chosen as it follows:
  \begin{align*}
  (1, -2)&: & 
  X
  &= - 1, 
  & 
  Y&= 0;
  &
  (0, 0)&: & 
  X
  &= - 1, 
  & 
  Y&= 0;
  &
  (0, -1)&: & 
  X
  &= 0, 
  & 
  Y&= 1;
  \\[2mm]
  (0, -2)&: & 
  X
  &= - 1, 
  & 
  Y&= 0;
  &
  (0, -3)&: & 
  X
  &= 0, 
  & 
  Y&= -1;
  &
  (-1, 0)&: & 
  X
  &= {- 1}, 
  & 
  Y&= 0;
  \\[2mm]
  (- 1, -2)&: & 
  X
  &= {1}, 
  & 
  Y&= 0;
  &
  (-2, 1)&: & 
  X
  &= 0, 
  & 
  Y&= 1;
  &
  (-2, 0)&: & 
  X
  &= 1, 
  & 
  Y&= 0;
  \\[2mm]
  (- 2, -1)&: & 
  X
  &=  0, 
  & 
  Y&= -1 ;
  &
  (-2, -2)&: & 
  X
  &= 1, 
  & 
  Y&= 0;
  &
  (-3, 0)&: & 
  X
  &= 1, 
  & 
  Y&= 0.
 \end{align*}
 \end{proposition}

\begin{remark} 
It is easy to verify that if in the solutions described in Proposition \text{\rm \ref{prop22}} we replace the numbers with the corresponding scalar matrices of size $m\times m$, then the resulting $2 m \times 2 m$-matrices define solutions of system \eqref{eq:resconditions} for $n=2 m.$
\end{remark}

Let us describe all solutions of system  \eqref{eq:resconditions}, for which  $[p,q]\ne 0$. Since $[p,q]^2 = 0$, it is possible to reduce the commutator to the block form  
\begin{gather} \label{eq2}
    K \stackrel{def}{=} [p, q]
    =
    \begin{blockarray}{cccc}
      \,\,\, m & m & k \,\,\, &  \\
    \begin{block}{(ccc)c}
      0 & \mathbb{I} & 0 & m  \\
      0 & 0 & 0 & m  \\
      0 & 0 & 0 & k  \\
    \end{block}
    \end{blockarray},
    \quad 
    k + 2 m = n,
\end{gather}
by a conjugation.
It follows from \eqref{mainR} that $p$ and $q$ have the following structure:
\begin{gather*}
    p
    =
    \begin{blockarray}{cccc}
      \,\,\, m & m & k \,\,\, &  \\
    \begin{block}{(ccc)c}
      \mu_1 \, \mathbb{I} & p_{12} & p_{13} & m  \\
      0 & \mu_2 \, \mathbb{I} & 0 & m  \\
      0 & p_{32} & p_{33} & k  \\
    \end{block}
    \end{blockarray},
    \qquad
    q
    =
    \begin{blockarray}{cccc}
      \,\,\, m & m & k \,\,\, &  \\
    \begin{block}{(ccc)c}
      \mu_3 \, \mathbb{I} & q_{12} & q_{13} & m  \\
      0 & \mu_4 \, \mathbb{I} & 0 & m  \\
      0 & q_{32} & q_{33} & k  \\
    \end{block}
    \end{blockarray}.
\end{gather*}
The normalization \eqref{eq2} and equations \eqref{eq:resconditions} provide the fact that the blocks $p_{33}, q_{33}$ commute and satisfy system \eqref{eq:resconditions}. Conjugating $p$ and $q$ with a matrix of the form ${\rm diag}(\mathbb{I}_{2m}, T_k)$, we reduce $p_{33}, q_{33}$ to a diagonal form (see Proposition \ref{prop2} and formula \eqref{eq:resJNF}). As a result, we obtain
\begin{equation} \label{anzat}
    \begin{multlined}
        p
        =
        \begin{blockarray}{ccccccc}
          \,\,\, m & m & k_1 & k_2 & k_3 & k_4 \,\,\, &  \\
        \begin{block}{(cccccc)c}
          \mu_1 \, \mathbb{I} & p_{12} & p_{13} & p_{14} & p_{15} & p_{16} & m  \\
          0 & \mu_2 \, \mathbb{I} & 0 & 0 & 0 & 0 & m  \\
          0 & p_{32} & - \mathbb{I} & 0 & 0 & 0 & k_1  \\
          0 & p_{42} & 0 & \mathbb{I} & 0 & 0 & k_2  \\
          0 & p_{52} & 0 & 0 & 0 & 0 & k_3  \\
          0 & p_{62} & 0 & 0 & 0 & 0 & k_4  \\
        \end{block}
        \end{blockarray},
        \hspace{5cm}
        \\
        q
        =
        \begin{blockarray}{ccccccc}
          \,\,\, m & m & k_1 & k_2 & k_3 & k_4 \,\,\, &  \\
        \begin{block}{(cccccc)c}
          \mu_3 \, \mathbb{I} & q_{12} & q_{13} & q_{14} & q_{15} & q_{16} & m  \\
          0 & \mu_4 \, \mathbb{I} & 0 & 0 & 0 & 0 & m  \\
          0 & q_{32} & - \mathbb{I} & 0 & 0 & 0 & k_1  \\
          0 & q_{42} & 0 & 0 & 0 & 0 & k_2  \\
          0 & q_{52} & 0 & 0 & \mathbb{I} & 0 & k_3  \\
          0 & q_{62} & 0 & 0 & 0 & 0 & k_4  \\
        \end{block}
        \end{blockarray}.
    \end{multlined}
\end{equation}

\begin{theorem}\label{theo2} For each point $(\alpha,\beta)\in \Sigma_0$ and any non-negative integers $m, k_1,k_2,k_3, k_4$ there exists a unique up to a conjugation solution of system \eqref{eq:resconditions}. It can be reduced to the form \eqref{anzat}, where $p_{1,i}=q_{1,i}=p_{i,2}=q_{i,2}=0$ \,for $ i=3,4,5,6$,\, and $p_{1,2}=X\, {\mathbb I}$, \, $q_{1,2}=Y\, {\mathbb I}$ \text{\rm{(}}the numbers $X$ and $Y$ are given in Proposition {\rm \ref{prop22}\text{\rm{)}}}.
 \end{theorem}

\begin{proof} We outline the proof for $\alpha = 0$, $\beta = -3$. In other cases, the proof is similar. 
The formula \eqref{muu} gives $\mu_1 = 0$, $\mu_2 = 1$, $\mu_3 = \mu_4 = 0$. Relations \eqref{eq:resconditions}, \eqref{eq2} are equivalent to the fact that
\begin{equation}
\begin{multlined}
    p
    =
    \begin{blockarray}{ccccccc}
      \,\,\, m & m & k_1 & k_2 & k_3 & k_4 \,\,\, &  \\
    \begin{block}{(cccccc)c}
      0 & p_{12} & q_{13} & p_{14} & 0 & 0 & m  \\
      0 & \mathbb{I} & 0 & 0 & 0 & 0 & m  \\
      0 & 2 q_{32} & - \mathbb{I} & 0 & 0 & 0 & k_1  \\
      0 & 0 & 0 & \mathbb{I} & 0 & 0 & k_2  \\
      0 & - q_{52} & 0 & 0 & 0 & 0 & k_3  \\
      0 & p_{62} & 0 & 0 & 0 & 0 & k_4  \\
    \end{block}
    \end{blockarray},
    \hspace{5cm}
    \\
    q
    =
    \begin{blockarray}{ccccccc}
      \,\,\, m & m & k_1 & k_2 & k_3 & k_4 \,\,\, &  \\
    \begin{block}{(cccccc)c}
      0 & - \mathbb{I} - q_{13} \, q_{32} + q_{15} \, q_{52} & q_{13} & 0 & q_{15} & 0 & m  \\
      0 & 0 & 0 & 0 & 0 & 0 & m  \\
      0 & q_{32} & - \mathbb{I} & 0 & 0 & 0 & k_1  \\
      0 & 0 & 0 & 0 & 0 & 0 & k_2  \\
      0 & q_{52} & 0 & 0 & \mathbb{I} & 0 & k_3  \\
      0 & 0 & 0 & 0 & 0 & 0 & k_4  \\
    \end{block}
    \end{blockarray}.
\end{multlined}
\end{equation}
After the conjugation $p\mapsto g p g^{-1}, \,q\mapsto g p g^{-1},$ where 
$$
    g
    = 
    \begin{blockarray}{ccccccc}
      \,\,\, m & m & k_1 & k_2 & k_3 & k_4 \,\,\, &  \\
    \begin{block}{(cccccc)c}
      \mathbb{I} & - p_{12} - 2 q_{13} \, q_{32} - q_{15} \, q_{52} & q_{13} & -p_{14} & -q_{15} & 0 & m  \\
      0 & \mathbb{I} & 0 & 0 & 0 & 0 & m  \\
      0 & - q_{32} & \mathbb{I} & 0 & 0 & 0 & k_1  \\
      0 & 0 & 0 & \mathbb{I} & 0 & 0 & k_2  \\
      0 & q_{52} & 0 & 0 & \mathbb{I} & 0 & k_3  \\
      0 & - p_{62} & 0 & 0 & 0 & \mathbb{I} & k_4  \\
    \end{block}
    \end{blockarray},
$$
we obtain 
\begin{gather}
    p
    =
    \begin{blockarray}{cccc}
      \,\,\, m & m & k \,\,\, &  \\
    \begin{block}{(ccc)c}
      0 & 0 & 0 & m  \\
      0 & \mathbb{I} & 0 & m  \\
      0 & 0 & p_{33} & k
      \\
    \end{block}
    \end{blockarray},
    \qquad
    q
    =
    \begin{blockarray}{cccc}
      \,\,\, m & m & k \,\,\, &  \\
    \begin{block}{(ccc)c}
      0 & - \mathbb{I} & 0 & m  \\
      0 & 0 & 0 & m  \\
      0 & 0 & q_{33} & k
      \\
    \end{block}
    \end{blockarray},
    \\
    \notag
    p_{33}
    = \diag (\,- \mathbb{I}_{k_1}, \mathbb{I}_{k_2}, 0_{k_3}, 0_{k_4}),
    \qquad 
    q_{33}
    = \diag (\,- \mathbb{I}_{k_1}, 0_{k_2}, \mathbb{I}_{k_3}, 0_{k_4}),
    \quad 
    k_1 + k_2 + k_3 + k_4 = k.
\end{gather}
\end{proof}

\subsubsection{Maximal solutions with commuting residues}\label{sec211}

Let us  consider first the maximal formal solutions of the form \eqref{y} under the condition $[p, q]=0$. In this case, the canonical form of the residues is given by the formula \eqref{eq:resJNF}.
Since the system \eqref{3rootN} admits transformations
$u \mapsto S \, u \, S^{-1}$, $v \mapsto S \, v \, S^{-1}$, where $S$ is an arbitrary non-degenerate matrix, the group $GL_n (\mathbb{C})$ acts on the pair of residues $p$, $q$. The maximum number of arbitrary parameters that can be contained in the residues is equal to the dimension of the orbit $\mathcal{O}$ of this action. Since it can be calculated as the difference between $n^2$ and the dimension of the stabilizer for the pair of matrices $p$ and $q$, we have:\, $\dim~\mathcal{O}~=~n^2 - \brackets{k_1^2 + k_2^2 + k_3^2 + k_4^2}$.

Other arbitrary parameters can appear in the solutions of the system \eqref{eq:matrixformofconditions} when $k$ belongs to the spectrum of the operator  $\mathcal{L}$. For any maximal solution \eqref{y} the total number of parameters must be equal to $2 n^2-1$. One more arbitrary parameter is $z_0.$

\begin{proposition} \label{thm:spectrumL} Suppose that the matrices  $p$ and $q$ have the form  \eqref{eq:resJNF}. Then the spectrum of the operator \eqref{eq:Loperator} belongs to the set 
\begin{gather*}
    \begin{aligned}
    \{
    \lambda_1
    &= - 2, 
    &
    \lambda_2
    &= 2,
    & 
    \lambda_3 
    &= - 1, 
    &
    \lambda_4
    &= 0, 
    &
    \lambda_5
    &= - \alpha, 
    &
    \lambda_6
    &= - \beta,
    &
    \lambda_7
    &= \alpha + 2,
    &
    \lambda_8
    &= \beta + 2,
    \end{aligned}
    \\
    \begin{aligned}
    \lambda_9
    &= 4 + \alpha + 2 \beta,
    &
    \lambda_{10}
    &= 4 + 2 \alpha + \beta,
    &
    \lambda_{11}
    &= 3 + \alpha + \beta,
    &
    \lambda_{12}
    &= - 2 - \alpha - 2 \beta,
    \end{aligned}
    \\
    \begin{aligned}\
    \lambda_{13}
    &= - 2 - 2 \alpha - \beta,
    &
    \lambda_{14}
    &= 1 + \alpha - \beta, 
    &
    \lambda_{15}
    &= 1 - \alpha + \beta, 
    &
    \lambda_{16}
    &= - 1 - \alpha - \beta
    \}.
    \end{aligned}
\end{gather*}
The dimensions of the corresponding eigenspaces are given by the formulas 
\begin{gather*}
    \begin{aligned}
    d_1
    &= d_2
    = k_1^2 + k_2^2 + k_3^2,
    & 
    d_3
    &= 2 \left(k_1 k_2 + k_1 k_3 + k_2 k_3 + k_1 k_4 + k_2 k_4 + k_3 k_4\right), 
    &
    d_4
    &= 2 k_4^2
    \end{aligned}
    \\
    \begin{aligned}
    d_5
    &= d_7
    = k_3 k_4, 
    & 
    d_6
    &= d_8
    = k_2 k_4,
    &
    d_9
    &= d_{12}
    = k_1 k_2,
    &
    d_{10}
    &= d_{13}
    = k_1 k_3,
    \end{aligned}
    \\
    \begin{aligned}
    d_{11}
    &= d_{16}
    = k_1 k_4,
    &
    d_{14}
    &= d_{15}
    = k_2 k_3.
    \end{aligned}
\end{gather*}
\end{proposition}
\begin{proof}
Let us represent $X$ and $Y$ as block $4 \times 4$-matrices:
\begin{gather}
    X
    =
    \begin{blockarray}{ccccc}
      & k_1 & k_2 & k_3 & k_4 \\
    \begin{block}{c(cccc)}
      k_1 & x_{11} & x_{12} & x_{13} & x_{14} \\
      k_2 & x_{21} & x_{22} & x_{23} & x_{24} \\
      k_3 & x_{31} & x_{32} & x_{33} & x_{34} \\
      k_4 & x_{41} & x_{42} & x_{43} & x_{44} \\
    \end{block}
    \end{blockarray}
    \,\, ,
    \qquad
    Y
    = 
    \begin{blockarray}{ccccc}
      & k_1 & k_2 & k_3 & k_4 \\
    \begin{block}{c(cccc)}
      k_1 & y_{11} & y_{12} & y_{13} & y_{14} \\
      k_2 & y_{21} & y_{22} & y_{23} & y_{24} \\
      k_3 & y_{31} & y_{32} & y_{33} & y_{34} \\
      k_4 & y_{41} & y_{42} & y_{43} & y_{44} \\
    \end{block}
    \end{blockarray}
    \, ,
\end{gather}
where the sizes of the blocks are prescribed  by formula \eqref{eq:resJNF}.
Then the system
$$
\brackets{
    \mathcal{L}
    - \lambda \, \mathbb{I}
    }
    \begin{pmatrix}
   X \\[1mm]
   Y
    \end{pmatrix}
    = 
   0
$$
is splitted into 16 independent subsystems with respect to pairs of blocks $x_{ij}, y_{ij}$. For example, for the blocks 
$x_{31}, y_{31}$ 
we have 
\begin{gather*}
    \left\{
    \begin{array}{rcr}
         \left(-1 - 2 \alpha - \lambda\right) x_{31} + \alpha y_{31} & = & 0,  
         \\[2mm]
         \left(2 + 2 \beta\right) x_{31} + \left(-2 - \beta - \lambda\right) y_{31} & = & 0. 
    \end{array}
    \right.
\end{gather*}
Equating the determinant of this system to zero, we find that  $\lambda = -1$ and $\lambda = -2-2 \alpha - \beta$ belong to the spectrum of $\mathcal{L}$. The dimension of the eigenspace corresponding to the second of the values of $\lambda$ is determined by the size of the corresponding block and is equal to $k_1 k_3$. The eigenvalue $\lambda_3 = -1$ occurs in other blocks as well. The sum $d_3$ of the numbers of elements in these blocks is given in Proposition \ref{thm:spectrumL}.
\end{proof}

\begin{definition}
Non-negative integer eigenvalues of the operator $\mathcal{L}$ are called {\it resonances}. 
The blocks that form the corresponding eigenspaces are called {\it resonance blocks}.
\end{definition}

\begin{lemma} \label{lem:max}
 Suppose that a solution \eqref{y} is maximal. Then  $k_1^2 + k_2^2 + k_3^2 = 1$ and all eigenvalues of the operator $\mathcal{L}$  except $\lambda_3 = - 1$ and $\lambda_1 = - 2$  are non-negative integers.
\end{lemma}

\begin{proof} If all eigenvalues except $\lambda_3 = - 1$ and $\lambda_1 = - 2$ are non-negative integers, then the possible number of arbitrary constants is equal to 
$$
2 n^2 - d_1 - d_3 + {\rm dim}\, {\mathcal O} = 2 n^2 - k_1^2-k_2^2-k_3^3.
$$
Since the number of parameters in a maximal solution must be equal to $2 n^2-1$ we proved the lemma.

\end{proof}
Thus, as in the scalar case, there are three types of possible maximal solutions:
\begin{itemize}
\item {\bf 1}: \qquad  $k_1=1, \quad k_2=k_3=0,\quad k_4=n-1$ ;
\item {\bf 2}: \qquad  $k_2=1, \quad k_1=k_3=0,\quad k_4=n-1$ ;
\item {\bf 3}: \qquad  $k_3=1, \quad k_1=k_2=0,\quad k_4=n-1$ .
\end{itemize}
In each of these cases, ${\rm dim}\,{\mathcal O} = 2 n -2$. Thus, for fixed $p$ and $q$, the coefficients of a maximum series must contain $2 n^2 - 2 n +1$ arbitrary constants.

\begin{proposition} \label{thm:albetcond_maxsol}
\phantom{}
\vspace{-0.2cm}
\begin{itemize}
    \item[\rm{i)}]
    A maximal solution of type {\bf 1} for system \eqref{3rootN} exists iff $\alpha+\beta$ is an integer satisfying the inequalities $-1 \geq \alpha+\beta\geq -3$.
    
    \vspace{-0.2cm}
    \item[\rm{ii)}]
    A maximal solution of type {\bf 2} exists iff $\beta$ is an integer and $0 \geq \beta\geq -2$. 
    
    \vspace{-0.2cm}
    \item[\rm{iii)}]
    A maximal solution of type {\bf 3} exists iff $\alpha$ is an integer and $0 \geq \alpha \geq -2$.
\end{itemize}
\end{proposition}
\begin{proof}
In Case {\bf1} (series of type {\bf 1}), the spectrum of the operator $\mathcal L$ is the set consisting of the following 6 numbers: $\lambda_1,\dots, \lambda_4,\, \lambda_{11} = 3 + \alpha + \beta$, and $\lambda_{16}= - 1 - \alpha - \beta$ (see Proposition \ref{thm:spectrumL}). The conditions of Lemma \ref{lem:max} hold if $\lambda_{11},\lambda_{16} \in \mathbb{Z}_{\geq 0}$. This means that $\alpha+\beta$ is an integer satisfying the inequalities $-1 \geq \alpha+\beta\geq -3$. Similarly, we show the necessity of the conditions of the proposition for solutions of the types {\bf 2} and {\bf 3}.

For series of type {\bf 1}, there exist three possibilities: $\alpha+\beta=-1,-2, -3$ with different sets of resonance blocks. To prove the sufficiency in Case {\bf 1}, we have to verify that in each case system \eqref{eq:matrixformofconditions} is compatible for all eigenvalues and the sum of dimensions of the resonant blocks equals $2 n^2 - n + 1$. This can be done by a direct calculation based on formula \eqref{eq:rhsFgamma}. The series of types {\bf 2} and {\bf 3} are considered in a similar way.
\end{proof}

\begin{remark}
Points from the set $\Sigma$ \text{\rm(}see Figure \text{\rm \ref{pic:albetplane}}\text{\rm)} are distinguished from other points $(\alpha,\beta)$ by the requirement that system \eqref{3rootN} has more than one maximal solution with commuting residues. For seven of these points, there are three different maximal solutions of types  {\bf1}-{\bf3}, and for the rest \text{\rm(}marked in Figure \text{\rm \ref{pic:albetplane}} with the dots surrounded by an orange rim\text{\rm)}, there are two. As will be shown below, they have one more maximal solution with non-commuting residues.
\end{remark}

\subsubsection{Maximal solutions with non-commuting \texorpdfstring{$p$}{p} and \texorpdfstring{$q$}{q}}\label{sec213}

If $(\alpha, \beta)\in \Sigma_0$ then there exist residues more general than \eqref{eq:resJNF} (see Theorem \ref{theo2}). Consider the case $\alpha=0$, \, $\beta=-3$ as an example. One can check that the eigenvalues for the operator   \eqref{eq:Loperator} are $\lambda=-2,\dots, 4$. The dimensions $d_{\lambda}$ of the corresponding eigenspaces are given by:
\begin{gather*}
    \begin{aligned}
    d_{-2}
    &= m^2 + m ( k_1 + 2 k_2 + k_3 ) + k_1^2 + k_2^2 + k_3^2 + k_1 k_2 + k_2 k_3,
    \end{aligned}
    \\[2mm]
    \begin{aligned}
    d_{-1}
    &= 3 m^2 + m (4 k_1 + 3 k_2 + 4 k_3 + 3 k_4)
    + 2 k_1 k_2 + 2 k_1 k_3 + 2 k_2 k_3 + 2 k_1 k_4 + 3 k_2 k_4 + 2 k_3 k_4,
    \end{aligned}
    \\[2mm]
    \begin{aligned}
    d_{0}
    &= 2 m^2 + m (k_1 + k_3 + 4 k_4) + 2 k_4^2 + k_1 k_4 + k_3 k_4,
    &&&&&&&&&
    d_{1}
    &= 2 k_1 k_3,
    \end{aligned}
    \\[2mm]
    \begin{aligned}
    d_2
    &= m^2 + m (k_1 + 2 k_2 + k_3)
    + k_1^2 + k_2^2 + k_3^2 + k_1 k_4 + k_3 k_4,
    \end{aligned}
    \\[2mm]
    \begin{aligned}
    d_3
    &= m^2 + m (k_2 + k_4) + k_2 k_4,
    &&&&&&&&&
    d_4 
    &= m (k_1 + k_3) + k_1 k_2 + k_2 k_3.
    \end{aligned}
\end{gather*}
The dimension of the stabilizer of the pair of residues $(p,q)$ is equal to
\begin{equation*}
    d 
    = m^2 + m (k_2 + k_4) + k_1^2 + k_2^2 + k_3^2 + k_4^2.
\end{equation*}
Since $\dim \mathcal{O} = n^2-d,$  the possible number of arbitrary constants $M$ in the formal solution  \eqref{y} is given by 
\begin{equation}\label{0,-3}
    M
    = 2 n^2 - d_{-2} - d_{-1} + \dim \mathcal{O}
    = 2 n^2 - m^2 - m (k_1 + 2 k_2 + k_3) - k_1^2 - k_2^2 - k_3^2 - k_2 (k_1 + k_3 + k_4).
\end{equation}

A similar calculation shows that for  $\alpha = 0$, $\beta = - 2$  the number $M$ of parameters is defined by the formula
\begin{equation}\label{0,-2}
    M = 2 n^2 - 2 m^2 - 2 m (k_2 + k_3) - k_1^2 - k_2^2 - k_3^2 - k_2 k_3.
\end{equation}
For the remaining points of $\Sigma_0$ the number $M$ can be found by means of the discrete symmetries \eqref{trr1} -- \eqref{trr3}.

The maximal solutions for which $m=0,$ are described in the previous section (see Section~\ref{sec211}). Let $m>0$.
In the case $\alpha=0$, \, $\beta=-3$ from formula \eqref{0,-3} it follows that $M$ may be equal $2 n^2-1$ only if $m=1$, $k_1=k_2=k_3=0$, $k_4=n-2$.

To prove that for $m = 1$ the formal solution \eqref{y} with the residues described in Theorem~\ref{theo2} is indeed maximal in the case $\alpha=0$, \, $\beta=-3$,
we represent the matrix coefficients of the formal series in the block form
\begin{gather*}
    x_k
    =
    \begin{blockarray}{cccc}
      \,\,\, 1 & 1 & n - 2 &  \\
    \begin{block}{(ccc)c}
      x_{k, 11} & x_{k, 12} & x_{k, 13} & 1  \\
      x_{k, 21} & x_{k, 22} & x_{k, 23} & 1  \\
      x_{k, 31} & x_{k, 32} & x_{k, 33} & n - 2  \\
    \end{block}
    \end{blockarray}
    , \qquad
    y_k
    =
    \begin{blockarray}{cccc}
      \,\,\, 1 & 1 & n - 2 &  \\
    \begin{block}{(ccc)c}
      y_{k, 11} & y_{k, 12} & y_{k, 13} & 1  \\
      y_{k, 21} & y_{k, 22} & y_{k, 23} & 1  \\
      y_{k, 31} & y_{k, 32} & y_{k, 33} & n - 2  \\
    \end{block}
    \end{blockarray}.
\end{gather*}
The resonances in these cases are $k = 0, 1, 2, 3$. Considering the recursion relations \eqref{eq:coeffconditions} for these $k$, we find that the resonance blocks are $x_{0, 11}$, $x_{0, 13}$, $x_{0, 31}$, $x_{0, 33}$, $y_{0, 11}$, $y_{0, 13}$, $y_{0, 31}$, $y_{0, 33}$, $x_{2, 22}$, $x_{3, 21}$, $x_{3, 23}$. 
A straightforward computation shows that system \eqref{eq:matrixformofconditions} is compatible for the resonance values of $k$ and the blocks listed above are arbitrary.
The sum of dimensions of the resonance blocks is equal to $2 n^2 - 3 n + 2$.  The above formula for the dimension of the orbit gives $\dim~\mathcal{O}~=~3 n-3$ and we get the maximal number $2 n^2-1$ of parameters in the series \eqref{y}. 
	
From the formula \eqref{0,-2}, it follows that in the case  $\alpha=0$, \, $\beta=-2$, there are no maximal solutions with $m>0$.

Putting together the results of Sections \ref{sec211} and \ref{sec213}, we arrive at the following statement:

\begin{theorem}\label{theo4}
System \eqref{3rootN} has three different maximal solutions of the form \eqref{y} iff $(\alpha,\beta) \in \Sigma$. For points in Figure \text{\rm{\ref{pic:albetplane}}} that have an orange rim, two of these solutions have commuting residues, and for the third the residues do not commute. For the remaining seven points, the residues commute with each other for all three maximal solutions.
\end{theorem}

\section{Inhomogeneous integrable systems of \texorpdfstring{$\PIV$}{P4} type}
\label{sec:P4tailedsystems}

In the scalar case, the procedure we use is as follows.
Suppose we want to find linear integrable deformations of a homogeneous system \eqref{Ptype} of the form
\begin{equation}\label{Ptypedef}
\left\{
\begin{array}{lcl}
u' &=& -u^2 + 2 u v - 2 z u + b_1 u + b_2 v + c_1, \\[2mm]
v' &=& -v^2 + 2 u v + 2 z v + b_3 v + b_4 u + c_2,
\end{array}
\right. \qquad b_i, c_i \in{\mathbb C}, 
\end{equation}
satisfying the \PPainleve--Kovalevskaya test to ``regain''\, system \eqref{syscal}. 

To reconstruct system \eqref{syscal}, it is not sufficient to require that system \eqref{Ptypedef} possesses one of the solutions of the type {\bf 1}, {\bf 2}, or {\bf 3}, containing the maximum possible number of arbitrary constants.

When restoring linear terms, the following principle works efficiently:
\begin{itemize} \item {\it All maximal solutions of a homogeneous system have to allow prolongation to solutions of the inhomogeneous system while preserving the property of their maximality. }
\end{itemize}

In particular, in the scalar example discussed above, we need to require the existence of maximal deformations of all three solutions, and then the answer will only be the system 
\begin{gather}\label{sysshift}
    \begin{cases}
    u'
    &= - u^2 
    + 2 u v - 2 z u
    + 2 \tau u
    + c_1
    ,
    \\[2mm]
    v'
    &= - v^2 + 2 u v + 2 z v
    - 2 \tau v
    + c_2
    .
    \end{cases}
\end{gather}
The parameter $\tau$ can be removed by a shift of $z$.

Below we use this procedure to find integrable linear deformations \eqref{3rootTail} of homogeneous matrix systems \eqref{3rootN}, corresponding to the pairs $(\alpha, \beta)$, shown by the red dots in Figure \ref{pic:albetplane}. For the remaining points, the answer can be obtained using transformations \eqref{trr1} -- \eqref{trr3}.

For diagonal residues $p$ and $q$ a maximal solution has to contain $2 n^2 - 2 n + 1$ arbitrary constants. For us to be able to add to this number the dimension $2 n - 2$ of the orbit ${\mathcal O}$, such a maximal solution must exist for arbitrary matrices $p$, $q \in {\mathcal O}$. Of course, we can reduce these matrices to a diagonal form by a conjugation, but all matrix coefficients $b_i$, $c_i$ will also be conjugated. Therefore, if we obtain some condition for these coefficients, it must also be satisfied when replacing $b_i \mapsto S \, b_i \, S^{-1}$, $c_i \mapsto S \, c_i \, S^{-1},$ where $S$ is an arbitrary non-degenerate matrix. In other words, the conditions have to be $GL_n$--invariant.
 
\begin{remark} \label{invariance}
If we have discovered that for some matrix $A$ the condition $a_{i,j} = 0$ is satisfied for $i \ne j$, then $GL_n$--invariance implies $A = \gamma\, \mathbb{I}$, $\gamma \in \mathbb{C}$.
\end{remark}

\subsection{Linear deformation of system \eqref{3rootN} with \texorpdfstring{$\alpha = \beta =-1$}{alpha = beta = -1}}

Recall that the homogeneous system \eqref{3rootN} in this case possesses maximal solutions of different three types (see Theorem~\ref{theo4}). The normal Jordan forms of their residues $p$ and $q$ are given~by 
\begin{align}
    \textbf{1} 
    &: 
    &
    p
    &= \diag \brackets{-1, \, 0_{n - 1}},
    &
    q
    &= \diag \brackets{-1, \, 0_{n - 1}};
    \\[1mm]
    \textbf{2} 
    &: 
    &
    p
    &= \diag \brackets{\phantom{-}1, \, 0_{n - 1}},
    &
    q
    &= \diag \brackets{\phantom{-}0, \, 0_{n - 1}};
    \\[1mm]
    \textbf{3} 
    &: 
    &
    p
    &= \diag \brackets{\phantom{-}0, \, 0_{n - 1}},
    &
    q
    &= \diag \brackets{\phantom{-}1, \, 0_{n - 1}}
\end{align}
and the coefficients $x_k$ and $y_k$ of the formal solutions \eqref{y} contain $2 n^2 - 2 n + 1$ arbitrary constants. We represent these coefficients as block $2 \times 2$-matrices:
\begin{align}
    x_{k}
    &= 
    \begin{blockarray}{ccc}
      1 & n - 1 &  \\
    \begin{block}{(cc)c}
      x_{k, 11} & x_{k, 12} & 1  \\
      x_{k, 21} & x_{k, 22} & n - 1  \\
    \end{block}
    \end{blockarray},
    &
    y_{k}
    &= 
    \begin{blockarray}{ccc}
      1 & n - 1 &  \\
    \begin{block}{(cc)c}
      y_{k, 11} & y_{k, 12} & 1  \\
      y_{k, 21} & y_{k, 22} & n - 1  \\
    \end{block}
    \end{blockarray}.
\end{align}
In all three cases the operator ${\mathcal L} - k \, \mathbb{I}$ is degenerate for $k = 0,1,2$. It follows from \eqref{eq:matrixformofconditions} that the eigenspaces for these resonance eigenvalues of the operator $\mathcal L$ consist of the blocks  $x_{0, 22}$, $y_{0, 22}$, $x_{1, 12}$, $x_{1, 21}$, $x_{2,11}$.

For the corresponding  inhomogeneous system \eqref{3rootTail} the same blocks must remain resonant. This means that system \eqref{eq:matrixformofconditions} has to be compatible for $k = 0,1,2$, which leads to conditions on the coefficients 
\begin{align}
    b_i
    &= 
    \begin{blockarray}{ccc}
      1 & n - 1 &  \\
    \begin{block}{(cc)c}
      b_{i, 11} & b_{i, 12} & 1  \\
      b_{i, 21} & b_{i, 22} & n - 1  \\
    \end{block}
    \end{blockarray},  \qquad 
    &
    c_i
    &= 
    \begin{blockarray}{ccc}
      1 & n - 1 &  \\
    \begin{block}{(cc)c}
      c_{i, 11} & c_{i, 12} & 1  \\
      c_{i, 21} & c_{i, 22} & n - 1  \\
    \end{block}
    \end{blockarray},
    &
    i 
    &= 1, \dots, 5,
\end{align} 
of the system.
We call these compatibility conditions {\it resonance conditions} at the eigenvalue~$k$.

In the relation \eqref{eq:matrixformofconditions} the operator $\mathcal L$ in the left hand sides is the same as in the homogeneous case while the right hand sides  $\begin{pmatrix} f_1 \brackets{x_k, y_k}, & f_2 \brackets{x_k, y_k} \end{pmatrix}^T$ are given by 
\begin{equation}
    \label{eq:rhs_tails}
    \begin{gathered}
    \begin{aligned}
    f_1 \brackets{x_k, y_k}
    &= f_{\alpha} \brackets{x_{k - 1}, y_{k - 1}}
    + 2 z_0 \, x_{k - 1}
    + 2 x_{k - 2} \, \brackets{1 - \delta_{0,k}}
    \\
    & \qquad
    - b_1 \, x_{k - 1}
    - x_{k - 1} \, b_2
    - b_3 \, y_{k - 1}
    - y_{k - 1} \, b_4
    - b_5 \, \delta_{1,k},
    \end{aligned}
    \\[2mm]
    \begin{aligned}
    f_2 \brackets{x_k, y_k}
    &= f_{\beta} \brackets{y_{k - 1}, x_{k - 1}}
    - 2 z_0 \, y_{k - 1}
    - 2 y_{k - 2} \, \brackets{1 - \delta_{0,k}}
    \\
    & \qquad
    - c_1 \, y_{k - 1}
    - y_{k - 1} \, c_2
    - c_3 \, x_{k - 1}
    - x_{k - 1} \, c_4
    - c_5 \, \delta_{1,k},
    \end{aligned}
    \end{gathered}
\end{equation}
where $\delta$ is the Kronecker delta, $x_{-1} = p$, $y_{-1} = q$, and the function $f_{\gamma} \brackets{x_k, y_k}$ is defined by \eqref{eq:rhsFgamma}. Notice that the functions $f_i$ in \eqref{eq:rhs_tails} depend not only on the previous coefficients of the series \eqref{y}, but also on the matrix parameters $b_i, c_i$ of deformation \eqref{3rootTail}. 

We are going to find the resonance conditions for $k=0,1,2$. If they are satisfied then the coefficients $x_k$, \, $y_k,$ where $k \ge 3$ are uniquely defined since the operator ${\mathcal L}- k \, \mathbb{I}$ in \eqref{eq:matrixformofconditions} is invertible. 

\medskip
\textbf{\textbullet \, Resonance conditions at $k = 0$}

When defining the matrices $x_0$, $y_0$ for series of all three types, the resonance conditions do not arise, the blocks $x_{0, 22}$, $y_{0, 22}$ turn out to be arbitrary, and the remaining blocks in  $x_0$, $y_0$ are uniquely defined.
 
\medskip
\textbf{\textbullet \, Resonance conditions at $k = 1$}

For the series of type \textbf{2}, two resonance conditions arise:
\begin{gather}
\begin{multlined}
    2 z_0 \, c_{4, 12}
    - \brackets{
    c_{3, 12} - c_{4, 12}
    } x_{0, 22}
    - \brackets{
    b_{2, 12} + c_{1, 12}
    } y_{0, 22}
    \\
    + \frac{1}{2} \left(
    - c_{3, 11} b_{2, 12} 
    + c_{3, 11} c_{2, 12} 
    + c_{4, 11} b_{2, 12} 
    + c_{4, 11} c_{2, 12}
    + 2 c_{1, 11} c_{4, 12}
    + 2 c_{4, 12} c_{2, 22} 
    \right.
    \\
    \left.
    \hspace{5cm}
    + 2 c_{3, 11} c_{4, 12}
    - 2 b_{2, 12} c_{4, 22}
    + 2 c_{4, 12} c_{4, 22}
    - 2 c_{5, 12}
    \right)
    = 0,
\end{multlined}
\\[2mm]
\begin{multlined}
    2 z_0 \, c_{3, 21}
    + x_{0, 22} \brackets{
    c_{3, 21} - c_{4, 21}
    } 
    - y_{0, 22} \brackets{
    b_{1, 21} + c_{2, 21}
    }
    \\
    + \frac{1}{2} \left(
    b_{1, 21} c_{3, 11}
    + c_{1, 21} c_{3, 11}
    + 2 c_{1, 22} c_{3, 21}
    + 2 c_{3, 21} c_{2, 11} 
    - 2 c_{3, 22} b_{1, 21} 
    + 2 c_{3, 22} c_{3, 21} 
    \right.
    \\
    \left.
    \hspace{5cm}
    - b_{1, 21} c_{4, 11}
    + c_{1, 21} c_{4, 11}
    + 2 c_{3, 21} c_{4, 11}
    - 2 c_{5, 21}
    \right)
    = 0.
\end{multlined}
\end{gather}

Due to the arbitrariness of the point $z_0$ and the blocks $x_{0, 22}$, $y_{0, 22}$, these relations are equivalent to a system of eight equations. From the six simplest equations and the $GL_n$--invariance of the conditions (see Remark \ref{invariance}), it follows that
\begin{align} \label{eq:cond2_1}
    c_1
    &= - b_2 + \gamma_1 \, \mathbb{I},
    &
    c_2
    &= - b_1 + \gamma_2 \, \mathbb{I},
    &
    c_3
    &= \gamma_3 \, \mathbb{I},
    &
    c_4
    &= \brackets{\gamma_3 + \gamma_4} \, \mathbb{I},
    &
    \gamma_i
    &\in {\mathbb C}.
\end{align}
Taking into account relations \eqref{eq:cond2_1}, we find that the remaining two equations are equivalent to the relation
\begin{align} \label{eq:cond2_2}
    &&
    c_5
    &= - \brackets{\gamma_3 + \frac{1}{2} \gamma_4} \brackets{b_1 + b_2} 
    + \gamma_5 \, \mathbb{I},
    &
    \gamma_5
    &\in {\mathbb C}.
    &&
\end{align}
If the matrices $b_i$, $c_i$ satisfy the conditions \eqref{eq:cond2_1} -- \eqref{eq:cond2_2}, then the coefficients $x_1$, $y_1$ in the formal solution of  type \textbf{2} has two arbitrary blocks $x_{1, 12}$ and $x_{1, 21}$.

For the solution of type \textbf{3}, two resonance conditions also appear:
\begin{gather} \label{eq:cond334}
\begin{multlined}
    - 2 z_0 \, b_{4, 12}
    - \brackets{
    b_{1, 12} + c_{2, 12}
    } x_{0, 22}
    - \brackets{
    b_{3, 12} - b_{2, 12}
    } y_{0, 22}
    \\
    + \frac{1}{2} \left(
    b_{3, 11} b_{2, 12}
    + b_{4, 11} b_{2, 12}
    + 2 b_{1, 11} b_{4, 12}
    + 2 b_{4, 12} b_{2, 22}
    + 2 b_{3, 11} b_{4, 12}
    + 2 b_{4, 12} b_{4, 22}
    \right.
    \\
    \left.
    \hspace{5cm}
    - b_{3, 11} c_{2, 12}
    + b_{4, 11} c_{2, 12}
    - 2 c_{2, 12} b_{4, 22}
    - 2 b_{5, 12}
    \right)
    = 0,
\end{multlined}
\\[2mm]
\label{eq:cond343}
\begin{multlined}
    - 2 z_0 \, b_{3, 21}
    - x_{0, 22} \brackets{
    b_{2, 21} + c_{1, 21}
    }
    + y_{0, 22} \brackets{
    b_{3, 21} - b_{4, 21}
    }
    \\
    + \frac{1}{2} \left(
    b_{1, 21} b_{3, 11}
    + 2 b_{1, 22} b_{3, 21}
    + 2 b_{3, 21} b_{2, 11}
    + 2 b_{3, 22} b_{3, 21}
    + b_{1, 21} b_{4, 11}
    + 2 b_{3, 21} b_{4, 11}
    \right.
    \\
    \left.
    \hspace{5cm}
    + c_{1, 21} b_{3, 11}
    - 2 b_{3, 22} c_{1, 21}
    - c_{1, 21} b_{4, 11}
    - 2 b_{5, 21}
    \right)
    = 0.
\end{multlined}
\end{gather}
Reasoning as above, we derive 8 algebraic equations from the conditions \eqref{eq:cond334} -- \eqref{eq:cond343}. The $GL_n$--invariant solution of the simplest six of them has the form
\begin{align} \label{eq:cond3_1}
    c_1
    &= - b_2 + \delta_1 \, \mathbb{I},
    &
    c_2
    &= - b_1 + \delta_2 \, \mathbb{I},
    &
    b_3
    &= \delta_3 \, \mathbb{I},
    &
    b_4
    &= \brackets{\delta_3 + \delta_4} \, \mathbb{I}, 
    &
    \delta_i
    &\in {\mathbb C},
\end{align}
and the two remaining equations lead to  
\begin{align} \label{eq:cond3_2}
    &&
    b_5
    &= \brackets{\delta_3 + \frac{1}{2} \delta_4} \brackets{b_1 + b_2} 
    + \delta_5 \, \mathbb{I}, 
    &
    \delta_5 
    &\in {\mathbb C}.
    &&
\end{align}
If the conditions \eqref{eq:cond3_1} -- \eqref{eq:cond3_2} for the matrices $b_i$, $c_i$ hold, then the series of type \textbf{3} has two arbitrary blocks $x_{1, 12}$ and $x_{1, 21}$.

Taking together, the conditions \eqref{eq:cond2_1} -- \eqref{eq:cond2_2} and \eqref{eq:cond3_1} -- \eqref{eq:cond3_2} are equivalent to the relations
\begin{gather*}
    \begin{aligned}
    c_1
    &= - b_2 + \gamma_1 \, \mathbb{I},
    &
    c_2
    &= - b_1 + \gamma_2 \, \mathbb{I},
    &
    c_3
    &= \gamma_3 \, \mathbb{I},
    &
    c_4
    &= \brackets{\gamma_3 + \gamma_4} \, \mathbb{I},
    \end{aligned}
    \\[1mm]
    \begin{aligned}
    c_5
    &= - \brackets{\gamma_3 + \frac{1}{2} \gamma_4} \brackets{b_1 + b_2} 
    + \gamma_5 \, \mathbb{I},
    \end{aligned}
    \\[1mm]
    \begin{aligned}
    b_3
    &= \delta_3 \, \mathbb{I},
    &
    b_4
    &= \brackets{\delta_3 + \delta_4} \, \mathbb{I},
    &
    b_5
    &= \brackets{\delta_3 + \frac{1}{2} \delta_4} \brackets{b_1 + b_2} 
    + \delta_5 \, \mathbb{I}.
    \end{aligned}
\end{gather*}
Substituting them into the resonance conditions at $k=1$ for the solution of type \textbf{1}, we obtain that it has two arbitrary blocks $x_{1, 12}$ and $x_{1, 21}$ iff
\begin{align} \label{eq:cond1_1}
    &&
    \LieBrackets{b_1, b_2}
    &= \frac{1}{2} \brackets{\gamma_1 + \gamma_2} \brackets{b_1 + b_2} 
    + \varepsilon \, \mathbb{I}, 
    &
    \varepsilon 
    &\in {\mathbb C}.
    &&
\end{align}

\medskip
\textbf{\textbullet \, Resonance conditions at $k = 2$}

Each of the series of type \textbf{2} and \textbf{3} has one resonance condition. They are given by
\begin{multline*}
    (\gamma_1 + \gamma_2) (2 \gamma_3 + \gamma_4) z_0
    + \frac{1}{4} (8 \gamma_3 
    - 4 \gamma_1^2 \gamma_3 
    - 8 \gamma_1 \gamma_2 \gamma_3 
    - 4 \gamma_2^2 \gamma_3 
    - 4 \gamma_1 \gamma_3^2 
    \\
    - 4 \gamma_2 \gamma_3^2 
    + 4 \gamma_4 
    - 2 \gamma_1^2 \gamma_4 
    - 4 \gamma_1 \gamma_2 \gamma_4 
    - 2 \gamma_2^2 \gamma_4 
    - 4 \gamma_1 \gamma_3 \gamma_4 
    - 4 \gamma_2 \gamma_3 \gamma_4 
    \\
    - \gamma_1 \gamma_4^2 
    - \gamma_2 \gamma_4^2 
    + 4 \gamma_1 \gamma_5 
    + 4 \gamma_2 \gamma_5)
    = 0,
\end{multline*}
and 
\begin{multline*}
    (\gamma_1 + \gamma_2) (2 \delta_3 + \delta_4) z_0
    + \frac{1}{4} (-8 \delta_3 
    - 4 \delta_4 
    - 4 \delta_3^2 \gamma_1 
    - 4 \delta_3 \delta_4 \gamma_1 
    - \delta_4^2 \gamma_1 
    \\
    + 4 \delta_5 \gamma_1 
    - 4 \delta_3^2 \gamma_2 
    - 4 \delta_3 \delta_4 \gamma_2 
    - \delta_4^2 \gamma_2 
    + 4 \delta_5 \gamma_2)
    = 0,
\end{multline*}
respectively. The corresponding system of four equations has two solutions:
\begin{align} \label{eq:cond231}
    \gamma_2
    &= - \gamma_1,
    &
    \gamma_4
    &= - 2 \gamma_3,
    &
    \delta_4
    &= - 2 \delta_3;
\end{align}
and
\begin{align} \label{eq:cond232}
    \gamma_5
    &= \delta_5 = 0,
    &
    \gamma_4
    &= - 2 \gamma_3,
    &
    \delta_4
    &= - 2 \delta_3.
\end{align}
If relations \eqref{eq:cond231} are satisfied, then the series of type \textbf{1}, \textbf{2}, \textbf{3} have an arbitrary block $x_{2, 11}$, i.e. they are maximal. In system \eqref{eq:cond1_1}, one have to put $\varepsilon = 0$, since the commutator of two matrices is a traceless matrix.

The system corresponding to solution \eqref{eq:cond231} has the form 
\begin{align} \label{eq:case1tail}
    &
    \left\{
    \begin{aligned}
    u'
    &= - u^2 + u v + v u - 2 z u + b_1 u + u b_2 + \gamma_1 \, \mathbb{I},
    \\
    v'
    &= - v^2 + v u + u v + 2 z v - b_2 v - v b_1 + \gamma_2 \, \mathbb{I},
    \end{aligned}
    \right.
    &
    \LieBrackets{b_1, b_2}
    &= 0,
    &
    \gamma_1, \gamma_2 
    &\in \mathbb{C}.
\end{align}

In the case of solution \eqref{eq:cond232}, the requirement for the existence of a maximal solution of type {\bf 1} leads to the relation
$$ (\gamma_1 + \gamma_2)^2 z_0 + \frac{1}{4} (\gamma_1 + \gamma_2) (-12 + \gamma_1^2 + 2 \gamma_1 \gamma_2 + \gamma_2^2) = 0.$$
This implies that $\gamma_1 = - \gamma_2$ and we arrive at a  particular case of the system \eqref{eq:case1tail}.

\begin{remark}
This system can be rewritten in the spirit of the paper
\text{\rm\cite{Retakh_Rubtsov_2010}} using the noncommutative independent variable $\bar z$:
\begin{align*}  
    &
    \left\{
    \begin{array}{lcl}
    u'
    &=& - u^2 + u v + v u + (k - 2) \, \bar zu - k\, u \bar z + \gamma_1,
    \\[2mm]
    v'
    &=& - v^2 + v u + u v + k\, \bar z v - (k - 2)\, v \bar z + \gamma_2.
    \end{array}
    \right.
\end{align*}
\end{remark}

Applying the following transformation of the form \eqref{tranconj}
\begin{align*}
    u 
    & \mapsto e^{z b_2} u e^{- z b_2},
    &
    v 
    & \mapsto e^{z b_2} v e^{- z b_2},
\end{align*}
and renaming $h = b_1 + b_2$, we reduce system \eqref{eq:case1tail} to the canonical form \ref{eq:case1tail_can}.

\medskip
\subsection{Case \texorpdfstring{$\alpha = 0,\, \beta =-3$}{alpha = 0, beta = -3}} \label{sec:case5_tail}

In Sections \ref{sec:case5_tail} and \ref{sec:case3_tail}  we  use the following shifts
\begin{align} \label{eq:tailtransform}
    &
    \begin{gathered}
    \begin{aligned}
        b_1 
        &\mapsto b_1 - 2 \gamma_1 \, \mathbb{I},
        &&&
        b_2
        &\mapsto b_2 - 2 \gamma_2 \, \mathbb{I},
        &&&
        c_1
        &\mapsto c_1 + 2 \gamma_1 \, \mathbb{I},
    \end{aligned}
    \\[2mm]
    \begin{aligned}
        c_2
        &\mapsto c_2 + 2 \gamma_2 \, \mathbb{I}, 
        &&&
        z
        &\mapsto z - \gamma_1 - \gamma_2,
    \end{aligned}
    \end{gathered} 
    &
    \gamma_i 
    &\in {\mathbb C},
\end{align}
under which the system \eqref{3rootN} is invariant, to simplify the matrix coefficients.

In contrast to the previous case, the homogeneous system \eqref{3rootN} has only two maximal formal solutions of  types \textbf{1} and \textbf{3} with the resonance blocks $x_{0, 12}$, $x_{0, 22}$, $y_{0, 22}$, $x_{2, 11}$, $y_{2, 21}$. We require that the same holds true for the deformation \eqref{3rootTail}.

Calculations similar to those described in the previous section lead to the system
\begin{gather} 
    \left\{
    \begin{array}{lcl}
    u' 
    &=& - u^2 + 2 u v - 2 z u - \LieBrackets{u, h_1} - 3 h_2 u + h_2 v + h_4,
    \\[2mm]
    v'
    &=& - v^2 + 3 u v - v u + 2 z v - \LieBrackets{v, h_1} + 3 v h_2 + h_3 u - 3 u h_2 + \brackets{2 h_4 + \frac{1}{2} h_2 h_3 + \gamma \, \mathbb{I}},
    \end{array}
    \right.
    \\
    \label{eq:case5tail}
    \\[-3mm]
    \begin{aligned}
    \LieBrackets{h_1, h_2}
    &= \LieBrackets{h_1, h_3}
    = \LieBrackets{h_2, h_3}
    = 0,
    \end{aligned}
    \\[1mm]
    \begin{aligned}
    \LieBrackets{h_4, h_2 - h_3}
    &= - 2 \brackets{h_2 - h_3},
    &
    \LieBrackets{h_4, 2 h_1 - 5 h_2}
    &= - 2 h_2,
    &
    \gamma \in \mathbb{C}.
    \end{aligned}
\end{gather}

System \eqref{eq:case5tail} can be simplified by a proper transformation of the form \eqref{tranconj}. It follows from the commutator relations that 
\begin{align*}
    e^{- \frac12 z (2 h_1 - 5 h_2)} h_4 e^{\frac12 z (2 h_1 - 5 h_2)}
    &= h_4 - \frac12 z [2 h_1 - 5 h_2, h_4]
    = h_4 - z h_2.
\end{align*}Using this formula, one can verify that the mapping
\begin{align*}
    u 
    & \mapsto e^{- \frac12 z (2 h_1 - 5 h_2)} 
    \brackets{u + \frac12 h_2} 
    e^{\frac12 z (2 h_1 - 5 h_2)},
    &
    v 
    & \mapsto e^{- \frac12 z (2 h_1 - 5 h_2)} 
    \brackets{v - h_2} 
    e^{\frac12 z (2 h_1 - 5 h_2)},
\end{align*}
reduces system \eqref{eq:case5tail} to  \ref{eq:case5tail_can}.

If all the coefficients are scalar, the system is contained in \cite[(21)]{adler2020}.

\subsection{Case \texorpdfstring{$\alpha = 0,\, \beta =-2$}{alpha = 0, beta = -2}} \label{sec:case3_tail}

Similar to the case $\alpha = \beta = -1$, the homogeneous system \eqref{3rootN} has three maximal formal solutions. The resonant blocks are $x_{0, 22}$, $y_{0, 22}$, $x_{1, 12}$, $y_{1, 21}$, $x_{2,11}$ in a series of type~\textbf{1}, $x_{0, 22}$, $y_{0, 21}$, $y_{0, 22}$, $x_{2, 11}$, $x_{2, 12}$ in a series of type~\textbf{2}, and $x_{0, 12}$, $x_{0, 22}$, $y_{0, 22}$, $x_{2, 11}$, $x_{2, 21}$ in a series of type~\textbf{3}. We require that the same blocks have to be resonance in the formal solution \eqref{y} for the deformation \eqref{3rootTail}.

After a renaming the coefficients the result can be written as
\begin{equation}
    \begin{gathered}
    \label{eq:case3tail}
    \left\{
    \begin{array}{lcl}
    u'
    &=& - u^2 + 2 u v - 2 z u + \LieBrackets{u, h_1} + 2 u h_2 - h_2 v + \brackets{h_3 + \gamma \, \mathbb{I}},
    \\[2mm]
    v' 
    &=& - v^2 + 2 u v + 2 z v + \LieBrackets{v, h_1} - 2 h_2 v + u h_2 + h_3,
    \end{array}
    \right.
    \\[2mm]
    \begin{aligned}
    \LieBrackets{h_1, h_2}
    &= 0,
    &
    \LieBrackets{h_3, h_2 + 2 h_1}
    &= - 2 h_2,
    &
    \gamma \in \mathbb{I}.
    \end{aligned}
    \end{gathered}
\end{equation}

In the case when all the coefficients are scalar, the system is also contained in \cite[(19)]{adler2020}.

System \eqref{eq:case3tail} is equivalent to the system from Conclusion in the paper \cite{Adler_Sokolov_2020_1}, which corresponds to the case $\alpha = -2$, $\beta = -1$. The equivalence is established by the composition of transformations \eqref{trr2}, \eqref{trr3}, and \eqref{tranconj}, supplemented by a scaling  and shift of the variables (see the footnote on page \pageref{footnote:scal}).

This system can be reduce to the canonical form \ref{eq:case3tail_can} by the mapping
\begin{align*}
    u 
    & \mapsto e^{\frac12 z (2 h_1 + h_2)} 
    \brackets{u - \frac12 h_2} 
    e^{- \frac12 z (2 h_1 + h_2)},
    &
    v 
    & \mapsto e^{\frac12 z (2 h_1 + h_2)} 
    \brackets{v + \frac12 h_2} 
    e^{- \frac12 z (2 h_1 + h_2)}
\end{align*}
and the subsequent renaming $h = h_3 + \frac34 h_2^2$.

\begin{remark} 
\phantom{}
\vspace{-0.2cm}
\begin{itemize}
    \item[\rm{i)}]
    In order to derive system \eqref{eq:case1tail} the existence of all three maximal solutions \eqref{y} of types {\bf 1}-{\bf 3} with commuting residues  $p$ and $q$ has to be used; 
    
    \vspace{-0.2cm}
    \item[\rm{ii)}]
    System \eqref{eq:case3tail} is defined by the requirement that two of solutions of types {\bf 1}-{\bf 3} are maximal. The third solution with commuting residues turns out to be maximal automatically; 
    
    \vspace{-0.2cm}
    \item[\rm{iii)}]
    To find  system \eqref{eq:case5tail} it is sufficient to require that two solutions with commuting residues are maximal. In addition, the system has the  maximal solution  with non-commuting residues described in Section \text{\rm\ref{sec213}}.
\end{itemize}
\end{remark} 

\section{Degeneracies to Painlev\'e-2 equations}\label{degeneracies}

In this section, we are going to establish relations between matrix systems \ref{eq:case1tail_can} -- \ref{eq:case5tail_can} and matrix $\PII$ equations 
\begin{alignat}{2}
    \label{P20}
    \tag*{$\text{P}_2^0$}
    & y''= 2 y^3 + x y + b y + y b + \alpha \, \mathbb{I}, 
    &\qquad
    & \alpha \in {\mathbb C},~~
    b \in \Mat_n (\mathbb{C}),
    \\[1mm]
    \label{P21}
    \tag*{$\text{P}_2^1$}
    & y''= [y,y'] + 2 y^3 + x y + a, 
    &
    & a \in \Mat_n (\mathbb{C}),
    \\[1mm]
    \label{P22}
    \tag*{$\text{P}_2^2$}
    & y''= 2 [y,y'] + 2 y^3 + x y + b y + y b + a, 
    &
    & a, b \in \Mat_n (\mathbb{C}),~~ 
    [b,a] = 2 b,
\end{alignat}
found in \cite{Adler_Sokolov_2020_1}. All these equations have the form
\begin{align}
    y'' 
    &= \kappa [y, y']
    + 2 y^3 + x y + b_1 y + y b_2 + a,
    &&&
    \kappa
    &\in \mathbb{C}, 
    &
    a, \,\, 
    b_1, \,\,
    b_2
    &\in \Mat_n (\mathbb{C}).
\end{align}
Such a matrix equation can be rewritten as the system of two equations
\begin{align} \label{eq:P2matsys}
     &\left\{
    \begin{array}{lcl}
         f'
         &=& - f^2 + g - \frac12 x \, \mathbb{I} - c_1,  
         \\[2mm]
         g'
         &=& 2 g f + \beta [g, f] 
         + c_2 f + f c_3 + c_4,
    \end{array}
    \right.
    &&&
    \beta
    &\in \mathbb{C}, 
    &
    c_i 
    &\in \Mat_n (\mathbb{C}).
\end{align}
Here \, $y (x) = f (x)$, \,\, $\kappa = - 1 - \beta$, \,\, $b_1 = c_2 + (2 + \beta) \, c_1$, \,\, $b_2 = c_3 - \beta \, c_1$, \, and \, $a = c_4 - \frac12 \, \mathbb{I}$.

Consider first the scalar $\PIV$ system \eqref{syscal}. The following transformation:
\begin{gather}
    \label{treps}
    \begin{aligned}
    z 
    &= \frac14 \varepsilon^{-3} - \varepsilon \, x, \quad 
    &
    u (z)
    &= - \frac14 \varepsilon^{- 3} - \varepsilon^{-1} \, f (x), \quad 
    &
    v (z)
    &= - 2 \varepsilon \, g (x),
    \end{aligned}
    \\
    \notag
    \begin{aligned}
    c_1
    &= - \frac{1}{16} \varepsilon^{-6},
    &
    c_2
    &= 2 \theta,
    &
    \theta 
    &\in \mathbb{C},
    \end{aligned}
\end{gather}
brings the system to the form
\begin{gather*}
    \left\{
    \begin{array}{lcl}
    f'
    &=& 2 \varepsilon^2 \brackets{2 f g - x f} - f^2 + g - \frac12 x,
    \\[2mm]
    g'
    &=& 2 \varepsilon^{2}
    \brackets{ - g^2 + x g} + 2 f g + \theta
    .
    \end{array}
    \right.
\end{gather*}
Passing to the limit $\varepsilon \to 0,$ we obtain a system, which is equivalent to the Painlev\'e-2 equation 
\begin{equation*}
    y''
    = 2 y^3 + x y + \brackets{\theta - \frac12}
\end{equation*}
with respect to  $y (x) = f (x)$.

Let us demonstrate that all matrix Painlev\'e-2 equations found in \cite{Adler_Sokolov_2020_1} can be obtained from systems of Section \ref{sec:P4tailedsystems} by a similar limiting transition. 

The matrix system \eqref{3rootTail}, as a result of the same  as in the scalar case substitution \eqref{treps}, supplemented by the shift $b_5 \mapsto b_5 + \frac{1}{16} \varepsilon^{-6} \, \mathbb{I}$, takes the form
\begin{gather} \label{eq:sys4tosys2}
    \left\{
    \begin{array}{lcl}
    f'
    &=& 2 \varepsilon^3 \brackets{- b_3 g - g b_4}
    + 2 \varepsilon^2 \brackets{2 f g + \alpha [f, g] - x f + \frac12 b_5}
    + \varepsilon \brackets{- b_1 f - f b_2}
    \\[3mm]
    && \qquad \qquad
    + \brackets{- f^2 + g - \frac12 x \, \mathbb{I}}
    + \frac14 \varepsilon^{-1} \brackets{- b_1 - b_2},
    \\[3mm]
    g'
    &=& 2 \varepsilon^2 \brackets{- g^2 + x g}
    + \varepsilon \brackets{- c_1 g - g c_2}
    + \brackets{2 g f + \beta [g, f] + \frac12 c_5}
    \\[3mm]
    && \qquad \qquad
    + \frac12 \varepsilon^{-1} \brackets{- c_3 f - f c_4}
    + \frac18 \varepsilon^{-3} \brackets{- c_3 - c_4}.
    \end{array}
    \right.
\end{gather}
Let us consider the limits of this system in particular cases.

\subsection{Case \texorpdfstring{$\alpha = \beta = -1$}{alpha = beta = -1}}

For system \eqref{eq:case1tail} corresponding to this case, we have
\begin{gather*}
    \begin{aligned}
    b_3
    &= b_4
    = 0, 
    &
    b_5 
    &= \gamma_1 \, \mathbb{I},
    \end{aligned}
    \\[1mm]
    \begin{aligned}
    c_1
    &= - b_2,
    &
    c_2
    &= - b_1,
    &
    c_3
    &= c_4
    = 0,
    &
    c_5
    &= \gamma_2 \, \mathbb{I},
    &
    \gamma_i 
    &\in \mathbb{C},
    \end{aligned}
\end{gather*}
and two arbitrary matrices $b_1$, $b_2$ commute: $[b_1, b_2] = 0$.  Let $b_1 = b_2 = 2 \, \varepsilon \, b$, where $b~\in~\Mat_n (\mathbb{C})$. Then the limiting system \eqref{eq:sys4tosys2} has the form 
\begin{gather*}
    \left\{
    \begin{array}{lcl}
    f'
    &=& - f^2 + g - \frac12 x \, \mathbb{I} - b,
    \\[2mm]
    g'
    &=& g f + f g + \gamma \, \mathbb{I}.
    \end{array}
    \right.
\end{gather*}
The latter system is equivalent to the \ref{P20} equation with respect to $y (x) = f (x)$. 

\subsection{Case \texorpdfstring{$\alpha = 0$, $\beta = - 3$}{alpha = 0, beta = -3}}

In this case for system \eqref{eq:case5tail} we have
\begin{gather*}
    \begin{aligned}
    b_1 
    &= - 3 h_2 + h_1,
    &
    b_2
    &= - h_1,
    &
    b_3
    &= h_2,
    &
    b_4
    &= 0,
    &
    b_5
    &= h_4,
    \end{aligned}
    \\[1mm]
    \begin{aligned}
    c_1
    &= h_1,
    &
    c_2 
    &= 3 h_2 - h_1,
    &
    c_3
    &= h_3,
    &
    c_4 
    &= - 3 h_2,
    &
    c_5
    &= 2 h_4 + \frac12 h_3 h_2 + \gamma \, \mathbb{I},
    \end{aligned}
\end{gather*}
where the matrix coefficients are connected by the relations
\begin{align*}
    [h_1, h_2]
    &= [h_1, h_3]
    = [h_2, h_3]
    = 0,
    &
    [h_4, h_2 - h_3]
    &= - 2 (h_2 - h_3),
    &
    [h_4, 2 h_1 - 5 h_2]
    &= - 2 h_2.
\end{align*}
Let $a = h_4 + \frac12 \gamma \, \mathbb{I}$ and $h_1 = h_3 = 3 h_2$, $h_2 = - \frac43 \, \varepsilon \, b$, with $b \in \Mat_n (\mathbb{C})$. Then, passing to the limit $\varepsilon \to 0$ in \eqref{eq:sys4tosys2}, we obtain the system
\begin{gather*}
    \left\{
    \begin{array}{lcl}
    f'
    &=& - f^2 + g - \frac12 x \, \mathbb{I} - b,
    \\[2mm]
    g'
    &=& 3 f g - g f - 2 [f, b] + a,
    \end{array}
    \right.
    \quad 
    [a, b]
    = - 2 b,
\end{gather*}
which is equivalent to the \ref{P22} equation for $y (x) = f (x)$. 

\subsection{Case \texorpdfstring{$\alpha = 0$, $\beta = - 2$}{alpha = 0, beta = -2}}

The coefficients of the deformed system \eqref{eq:case3tail} are given by
\begin{gather*}
    \begin{aligned}
    b_1 
    &= - h_1,
    &
    b_2 
    &= h_1 + 2 h_2,
    &
    b_3 
    &= - h_2,
    &
    b_4 
    &= 0,
    &
    b_5
    &= h_3 + \gamma \, \mathbb{I},
    \end{aligned}
    \\[1mm]
    \begin{aligned}
    c_1 
    &= - h_1 - 2 h_2,
    &
    c_2 
    &= h_1,
    &
    c_3
    &= 0,
    &
    c_4
    &= h_2,
    &
    c_5 
    &= h_3,
    \end{aligned}
\end{gather*}
and satisfy the following commutation relations
\begin{align*}
    [h_1, h_2]
    &= 0,
    &
    [h_3, 2 h_1 + h_2]
    &= - 2 h_2.
\end{align*}
Let $a = \frac12 h_3$, $h_1 = 0$ and $h_2 \mapsto \varepsilon^{-4} h_2$. Then after taking the limit $\varepsilon \to 0$ in \eqref{eq:sys4tosys2} we get the system
\begin{gather} \label{eq:sys2case4}
    \left\{
    \begin{array}{lcl}
    f'
    &=& - f^2 + g - \frac12 x \, \mathbb{I},
    \\[2mm]
    g'
    &=& 2 f g + a,
    \end{array}
    \right.
\end{gather}
which is equivalent to the \ref{P21} equation with respect to $y(x) = f (x)$.  

The system \eqref{eq:sys2case4} is also equivalent to the equation 
\begin{gather} \label{eq:P34case4}
    w''
    = \frac12 \left(w' - a\right) w^{-1} \left(w' + a\right)
    + 2 w^2 - x w
\end{gather}
for $w (x) = g (x)$, that coincides with the \PPainleve-34 equation in the scalar case.

\subsubsection*{Acknowledgements}

The authors are grateful to V. Adler, A. Odesskii, V. Poberezhny and V. Roubtsov for useful discussions. The research of the second author was carried out under the State Assignment 0029-2021-0004 (Quantum field theory) of the Ministry of Science and Higher Education of the Russian Federation. The first author was partially supported by the International Laboratory of Cluster Geometry NRU HSE, RF Government grant № 075-15-2021-608.
  
%\nocite{*}
\bibliographystyle{plain}
\bibliography{bib}

\begin{thebibliography}{1}

\bibitem{adler2020}
V.~E. Adler.
\newblock Painlev{\'e} type reductions for the non-{A}belian {V}olterra
  lattices.
\newblock {\em Journal of Physics A: Mathematical and Theoretical},
  54(3):\href{https://doi.org/10.1088/1751--8121/abd21f}{035204}, 2021.
\newblock \href{https://arxiv.org/abs/2010.09021}{arXiv:2010.09021}.

\bibitem{Adler_Sokolov_2020_1}
V.~E. Adler and V.~V. Sokolov.
\newblock On matrix {P}ainlev{\'e} {II} equations.
\newblock {\em Theoret. and Math. Phys.},
  207(2):\href{https://doi.org/10.4213/tmf10027}{188--201}, 2021.
\newblock \href{https://arxiv.org/abs/2012.05639}{arXiv:2012.05639}.

\bibitem{Balandin_Sokolov_1998}
S.~P. Balandin and V.~V. Sokolov.
\newblock On the {P}ainlev{\'e} test for non-{A}belian equations.
\newblock {\em Physics letters A},
  246(3-4):\href{https://www.sciencedirect.com/science/article/abs/pii/S0375960198003363?via\%3Dihub}{267--272},
  1998.

\bibitem{Retakh_Rubtsov_2010}
V.~S. Retakh and V.~V. Rubtsov.
\newblock Noncommutative {T}oda {C}hains, {H}ankel {Q}uasideterminants and
  {P}ainlev{\'e} {II} {E}quation.
\newblock {\em Journal of Physics. A, Mathematical and Theoretical},
  43(50):\href{https://iopscience.iop.org/article/10.1088/1751--8113/43/50/505204}{505204},
  2010.
\newblock \href{https://arxiv.org/abs/1007.4168}{arXiv:1007.4168}.

\bibitem{SW2}
V.~V. Sokolov and T.~Wolf.
\newblock Non-commutative generalization of integrable quadratic {ODE} systems.
\newblock {\em Letters in Mathematical Physics},
  110(3):\href{https://link.springer.com/article/10.1007\%2Fs11005--019--01229--0}{533--553},
  2020.

\bibitem{veselov1993}
A.~P. Veselov and A.~B. Shabat.
\newblock Dressing chains and the spectral theory of the {S}chr{\"o}dinger
  operator.
\newblock {\em Functional Analysis and Its Applications},
  27(2):\href{https://doi.org/10.1007/BF01085979}{81--96}, 1993.

\end{thebibliography}

\end{document}